\documentclass[12pt]{amsart}

\usepackage{amsfonts,amsthm,amsmath,amssymb,amscd,mathrsfs, enumerate}
 \numberwithin{equation}{section}
\usepackage{graphics}
\usepackage{indentfirst}
\usepackage{cite}
\usepackage{latexsym}
\usepackage[dvips]{epsfig}
\usepackage{color}
\usepackage{bookmark}
\usepackage{esint}
\usepackage{enumitem}
\hypersetup{colorlinks=true,citecolor=blue,linkcolor=red,urlcolor=blue}

\newtheorem{theorem}{Theorem}[section]
\newtheorem{lemma}[theorem]{Lemma}
\newtheorem{pro}[theorem]{Proposition}
\theoremstyle{remark}
\newtheorem{remark}{Remark}[section]

\renewcommand{\div}{\operatorname{div}}
\DeclareMathOperator{\supp}{supp}
\newcommand{\Bv}{\boldsymbol{v}}
\newcommand{\Bu}{\boldsymbol{u}}
\newcommand{\Be}{\boldsymbol{e}}

\newcommand{\Bf}{\boldsymbol{f}}
\newcommand{\BG}{\boldsymbol{G}}
\newcommand{\BmcF}{\boldsymbol{\mathcal{F}}}
\newcommand{\Bw}{\boldsymbol{w}}

\begin{document}

	\title
	[Self-similar solutions with external forces]
	{Self-similar solutions to the steady Navier-Stokes equations in arbitrarily high dimensions}

	\author{Changfeng Gui}
	\address{Department of Mathematics, Faculty of Science, University of Macau, Taipa, Macau}
	\email{changfenggui@um.edu.mo}

	\author{Hao Liu}
	\address{Department of Mathematics, Faculty of Science, University of Macau, Taipa, Macau}
	\email{haoliu@um.edu.mo}

	\author{Yun Wang}
	\address{School of Mathematical Sciences, Center for Dynamical Systems and Differential Equations, Soochow University, Suzhou, China}
	\email{ywang3@suda.edu.cn}

	\author{Chunjing Xie}
	\address{School of Mathematical Sciences,
		Ministry of Education Key Laboratory of Scientific and Engineering Computing,
		and CMA-Shanghai, Shanghai Jiao Tong University, 800 Dongchuan Road, Shanghai, China}
	\email{cjxie@sjtu.edu.cn}

\keywords{Navier-Stokes equations, self-similar solutions, large external forces, existence, uniqueness}
\subjclass[2020]{35Q30, 35C06, 35A01, 35A02, 76D05}

\begin{abstract}
	We study $(-1)$-homogeneous solutions of the steady Navier-Stokes
	equations in $\mathbb R^n\setminus\{0\}$ with $(-3)$-homogeneous,
	locally Lipschitz forces. Previous work \cite{Hao25} established the existence of $(-1)$-homogeneous solutions without any smallness assumption in dimensions $4\leq n\leq16$. In this paper, we remove
	the dimension restriction and obtain the existence of self-similar solutions in all higher
	dimensions. The new  ingredient is a  dimension-independent
	a priori estimate for the total head pressure on the sphere, obtained by adapting the
	 source-flux method and the weighted identity of the pressure of Wang and Yang \cite{WangYang26} to the
	homogeneous setting. This  combines  with the structural relation
	 between the radial velocity and the total head pressure used in \cite{Hao25} to close the a priori  estimates for the solutions in all dimensions. 
\end{abstract}

	\maketitle

	\section{Introduction and main results}
	We consider the steady Navier-Stokes equations in dimension $n$ with external force $\Bf=(f_1,f_2,\ldots,f_n)$:
	\begin{align} \label{NS}
		-\Delta \Bu + (\Bu\cdot \nabla )\Bu + \nabla p=\boldsymbol {f}, \quad \div \Bu=0,
	\end{align}
	where  $\Bu=(u_1,u_2,\ldots,u_n)$ is the velocity field and  $p$ is the pressure. The existence and uniqueness theory for \eqref{NS} is a fundamental topic in partial differential equations; see \cite{Galdi11}.

\subsection{Finite energy solutions}
The existence theory  for solutions with finite Dirichlet energy  goes back to Leray \cite{leray1933etude}. Especially, Leray's arguments  give the existence of weak solutions belonging to $H^1_{loc}$  in any dimension (see also \cite[Chapter 2]{Tsai18}).
 For smooth external forces, energy weak solutions have interior regularity in dimension four \cite{Gerhardt79}, whereas in higher dimensions, this is a open problem. On the other hand, Frehse and R{\r u}\v{z}i\v{c}ka \cite{Frehse94, FrehseRuzicka94, FrehseRuzicka96}, Struwe \cite{Struwe95}  showed the existence of regular solutions on $\mathbb{R}^5$ for any smooth compactly supported forces. Li and Yang \cite[Theorem~1.1]{LiYang22} extended this existence result to  $\mathbb{R}^n$ 
 for $5\leq n\leq15$.  Very recently, Wang and Yang \cite[Theorem 1.1]{WangYang26} proved  the existence of regular solutions for all dimensions $n\geq16$.

\subsection{Existence and uniqueness of self-similar solutions}

A fundamental property of \eqref{NS} is its invariance under the scaling
\begin{equation}\label{eq:sca}
		\begin{aligned}
			\Bu(x)\mapsto \Bu_{\lambda}(x)=\lambda\Bu(\lambda x),\quad
			p(x)\mapsto p_{\lambda}(x)=\lambda^2p(\lambda x), \quad
			\Bf(x)\mapsto \Bf_{\lambda}(x)=\lambda^3\Bf(\lambda x),
		\end{aligned}
	\end{equation}
   for any $\lambda>0$. That is, if $(\Bu,p,\Bf)$ satisfies \eqref{NS}, then $(\Bu_\lambda,p_\lambda,\Bf_\lambda)$ defined in \eqref{eq:sca} also satisfies \eqref{NS}.
In particular, $(\Bu,p)$ is called scale invariant or self-similar if
$(\Bu_{\lambda},p_{\lambda})=(\Bu,p)$ for every $\lambda>0$, and in this case, the force  must also be scale invariant. Self-similar solutions play an important role in understanding the singularity and asymptotic behavior of general solutions of Navier-Stokes system \cite{KorolevSverak11,MiuraTsai12,bang2023rigidity}. Meanwhile, it is easy to see that  self-similar solutions do not have finite energy.

In the unforced case, $\Bf=0$, classical examples of self-similar solutions include the two-dimensional Jeffery-Hamel flows \cite{Jeffery01041915,Hamel17} and the three-dimensional Landau solutions \cite{Landau44, Tian98,Sverak11}. For three-dimensional axisymmetric $(-1)$-homogeneous solutions with singular rays, Li, Li, and Yan \cite{LiLiYan18} classified the no-swirl solutions and constructed solutions with swirl.
For $n\geq4$, every smooth  self-similar velocity field on $\mathbb R^n\setminus\{0\}$ is zero; see \cite{Tsai98b} and \cite[Theorem~3, Section~6]{Sverak11}.
Bang, Gui, Liu, Wang, and Xie \cite{bang2023rigidity} extended this result  to smooth  solutions on $\mathbb R^n\setminus\{0\}$ satisfying $|\Bu(x)|\leq C/|x|$, without assuming self-similarity or smallness of $C$.

For $\Bf\neq0$, there is  an extensive perturbative theory for
small stationary solutions in critical spaces; see
\cite{Kozono95, Phan13, Kaneko19, Tsurumi19}.
Without a size restriction in a scale-invariant space, the nonlinear term can no longer be handled by a perturbative argument.
In dimension $n=4$, Shi \cite[Theorem~1.1]{Shi18} proved the existence of a self-similar weak solution for every $(-3)$-homogeneous force in $L^{4/3,\infty}(\mathbb R^4)$, without a smallness assumption.

Recently, in \cite[Theorem 1.1]{Hao25}, the existence of self-similar
solutions to \eqref{NS} for arbitrary $(-3)$-homogeneous locally Lipschitz forces
was established for $4\leq n\leq16$. That work also obtained
similar existence result in all dimensions $n\geq4$ for purely radial forces with
nonnegative radial component \cite[Theorem~4.1]{Hao25}. The 
dimension restriction in the general case comes from the estimate
of the positive part of the total head pressure and
from the integrability exponent required to close the spherical
energy estimates.

The main contribution of this paper is a new a priori estimate which
enables us to remove the dimension restriction in \cite{Hao25}. Motivated by
the source-flux method and the weighted identity of the pressure 
in
\cite{WangYang26}, we adapt their argument from the finite-energy
setting to the present homogeneous framework. This yields a dimension-independent
	a priori estimate for the positive part of the total head pressure on the sphere.  This,  together  with the important structural relation
	 between the radial velocity and the total head pressure used in \cite{Hao25} (see Lemmas \ref{halfenergyestimate} and \ref{lem:urestim} below), closes the energy estimates for the solutions in all dimensions, and allows the
existence theory to be extended to arbitrarily high dimensions.

 We restrict attention to $n\geq6$, since the four- and five-dimensional cases are covered by \cite{Shi18,Hao25}. Our main theorem gives existence for arbitrary forces and small-force uniqueness. Before stating it, we fix our notation. For a homogeneous vector field $\Bv$ on $\mathbb{R}^n\setminus\{0\}$ and $\alpha\in(0,1)$, the norms $\|\Bv\|_{C(S^{n-1})}$ and $\|\Bv\|_{C^\alpha(S^{n-1})}$ refer to the norms for the functions restricted to the unit sphere. We set
\begin{equation*}
  \|\Bv\|_{\operatorname{Lip}(S^{n-1})}
  :=\|\Bv\|_{L^\infty(S^{n-1})}
  +\|\nabla^{S^{n-1}}\Bv\|_{L^\infty(S^{n-1})}.
\end{equation*}
This norm is equivalent to the usual Lipschitz norm on the sphere. For a homogeneous extension of fixed degree, it is also equivalent to
$\|\Bv\|_{L^\infty(S^{n-1})}+\|\nabla\Bv\|_{L^\infty(S^{n-1})}$.

Our main results are as follows.
	\begin{theorem}\label{thm:main}
	Let $n\geq6$, and let $\Bf$ be a $(-3)$-homogeneous force that is locally Lipschitz on $\mathbb{R}^n\setminus\{0\}$.
    \begin{itemize}
	\item[(i)] There exists at least one self-similar solution $\Bu$ to \eqref{NS} such that $\Bu\in C^{2,\alpha}_{\mathrm{loc}}(\mathbb{R}^n\setminus\{0\})$ for every $\alpha\in(0,1)$ and
		\begin{equation}\label{eq:C2est}
			\|\Bu\|_{C(S^{n-1})} + 	 	\|\nabla \Bu\|_{C(S^{n-1})} + \|\nabla^2 \Bu\|_{C^{\alpha}(S^{n-1})} \leq C,
		\end{equation}
		where $C>0$ depends only on $\alpha$, $n$, and $\|\Bf\|_{\operatorname{Lip}(S^{n-1})}$.
        \item[(ii)] If, in addition, $\Bf$ is smooth on $\mathbb{R}^n\setminus\{0\}$, then the solution obtained in (i) is smooth on $\mathbb{R}^n\setminus\{0\}$.
  \item [(iii)] There exists $\epsilon=\epsilon(n)>0$ such that if
  \[
  \|\Bf\|_{\operatorname{Lip}(S^{n-1})}\leq \epsilon,
  \]
  then the solution is unique among all $(-1)$-homogeneous solutions in $C^2(\mathbb{R}^n\setminus\{0\})$.
    \end{itemize}
\end{theorem}

	\begin{remark}
		For $6\leq n\leq16$, Theorem~\ref{thm:main} recovers the corresponding
		result of \cite{Hao25}. The  new range is $n\geq17$. Moreover, there might not exist solutions of \eqref{NS} if $n=2$ (see \cite[Section 5]{Hao25})  or $n=3$ (see \cite{Shi18}). Hence, Theorem \ref{thm:main} together with the results in \cite{Hao25, Shi18} gives a clear description of the existence  of self-similar solutions to \eqref{NS}.
\end{remark}

\subsection{Main ideas of the proof}\label{sec:mainidea}

The  proof is based on the Leray-Schauder degree argument
 in \cite{Hao25}; the main issue is therefore to obtain enough
a priori bound of solutions depending only on the force. Let
\[
H=\frac{|\Bu|^2}{2}+p
\]
be the total head pressure, and denote its positive part by
$H_+=\max\{0,H\}$. Let
\begin{equation}\label{defunifiedquantities}
	\mathcal E
	:=\int_{S^{n-1}}
	\left(|\nabla\Bu|^2+(n-4)|\Bu|^2\right)\,d\sigma,
	\quad
	\mathcal H:=\|H_+\|_{L^\theta(S^{n-1})},
	\quad
	N_f:=\|\Bf\|_{\operatorname{Lip}(S^{n-1})},
\end{equation}
where
\[
\theta=\frac{(n-2)(n-1)}{2(n-3)}>\frac{n}{2}.
\]
The spherical energy identity, together with the relation between the
radial velocity and the total head pressure observed in
\cite{Hao25} (see Lemmas \ref{halfenergyestimate} and \ref{lem:urestim} below), yields
\begin{equation}\label{eq:energyineq}
	\mathcal E
	\leq C N_f\mathcal E^{\frac{1}{2}}
	+C(\mathcal H+N_f)^{\frac{3}{2}},
\end{equation}
see also \eqref{eq:Eclosureunified} below.
Thus the main task is to control $\mathcal H$ and eventually close the 
a priori bound of $\mathcal E$, depending only on $N_f$.
The earlier argument in \cite{Hao25} estimates $\mathcal H$ by using
the equation of $H$ 
\begin{equation*}
    -\Delta \, H
	+ \Bu \cdot \nabla H
	= - \frac{1}{2} |\partial_i u_j - \partial_j u_i |^2
	+\Bf\cdot \Bu - \div \Bf \quad \text{in }\mathbb{R}^n \setminus \{0\},
\end{equation*}
which formally implies that
$H$ satisfies the maximum principle.  Here and below, repeated Cartesian indices are summed from $1$ to $n$.
  The way to utilize this maximum principle under the homogeneous setting is to test the equation by 
$H_+^{\frac{n-4}{2}}$, inspired by \cite{Tsai98b}.  However, the resulting integrability exponent  imposes an upper restriction on the dimension.

To remove this dimension
restriction, we rewrite the source term in divergence form and introduce
the Bernoulli source flux
\begin{equation}\label{eq:defF}
	\BmcF
	:=\nabla(-\Delta)^{-1}(|\Bf||\Bu|)+\Bf .
\end{equation}
Using the same  multiplier 
$H_+^{\frac{n-4}{2}}$ as in \cite{Hao25}, one can easily obtain that  (see Lemma \ref{lem:fluxBernoulli})
\[
\mathcal H
\leq C(n)\|\BmcF\|_{L^n(S^{n-1})}.
\]
Here comes the key new estimate (see Lemma \ref{lem:sourcefluxSS}),
\begin{equation}\label{eq:keyest}
    \|\BmcF\|_{L^n(S^{n-1})}
    \leq CN_f\left(1+N_f^{\frac{1}{2}}+\mathcal E^{\frac{1}{2}}+\mathcal H^{\frac{1}{2}}\right).
\end{equation}
Young's inequality then gives
\begin{equation}\label{eq:estforH+}
   \mathcal H
\leq C\bigl(
N_f+N_f^2+N_f\mathcal E^{\frac{1}{2}}
\bigr). 
\end{equation}
The key feature of \eqref{eq:estforH+} is that this  estimate is independent of the dimension $n$ modulo a constant $C$ and improves the dimension-dependent estimates of $\mathcal{H}$ in \cite[Lemma 2.3]{Hao25}.
The  proof of  \eqref{eq:keyest} adapts
the weighted identity of the pressure in
\cite{WangYang26} to the homogeneous setting. More precisely,  a localized
endpoint-weighted estimate for the pressure  gives Morrey control of the velocity on fixed annuli. This
combines with the $L^2$-BMO estimate of the Newtonian potential and the homogeneity of $\BmcF$, yielding the
required  $L^n$ bound.

Substituting \eqref{eq:estforH+} into \eqref{eq:energyineq} produces  a sublinear
power $\mathcal E^{\frac{3}{4}}$ on the right-hand side, which can be absorbed
by Young's inequality. This gives the desired a priori estimate in
every dimension $n\geq6$ and removes the dimensional restriction in
\cite{Hao25}. Once the energy estimates for $u$ are obtained, \eqref{eq:estforH+}  directly gives higher integrability of $H_+$ beyond the critical exponent $\frac{n}{2}$.
With these estimates available, the construction of solutions using the 
Leray-Schauder argument is the same as  \cite{Hao25}.

\subsection {Organization of  the paper}
The rest of this paper is organized as follows: we first recall some key estimates from \cite{Hao25}
in Section \ref{sec:A-priori}. We also give the bound of $ \mathcal{H}=\|H_+\|_{L^\theta(S^{n-1})}$  in terms of  the Bernoulli source flux $\BmcF$ in this section. In Section \ref{subsec:morrey}, we prove the desired bound of $\BmcF$ claimed in \eqref{eq:keyest}. Finally, we close the energy estimates and prove the main result in Section \ref{sec:mainresult}.

	\section{Some Preliminaries and Estimates for the total head pressure}\label{sec:A-priori}

Most of the estimates in this section appeared in \cite{Hao25}. They are
 valid for all dimensions.
 We first recall some important estimates    from \cite{Hao25}, then we bound the total head pressure in terms of the Bernoulli source flux $\BmcF$. 

\subsection{Pressure estimates and spherical energy estimates}
\label{subsec_EE}

	We first recall the estimates for the pressure  from \cite[Lemma~2.1]{Hao25}. These  estimates follow from the properties of the Riesz operator and the  singular integral bounds for homogeneous functions in \cite[Lemma~A.1]{Hao25}. We state these facts in the form used below, without repeating their proofs.
	\begin{lemma}[Pressure   estimates]\label{Lemmapressure}
Assume $n\geq6$, $\Bf$ is $(-3)$-homogeneous and locally Lipschitz on $\mathbb{R}^n\setminus\{0\}$. Let $(\Bu,p)$ solve \eqref{NS} in the sense of distributions, with $\Bu\in C^2(\mathbb{R}^n\setminus\{0\})$ being homogeneous of degree $-1$. Then $p\in C^1(\mathbb{R}^n\setminus\{0\})$,  is unique up to a constant, and can be chosen as $(-2)$-homogeneous. For this $p$ and every $1<\beta<\infty$, we have
\begin{align}
\label{eq:pressure-potential-estimates}
 \|p\|_{L^\beta(S^{n-1})}
 &\leq C(n,\beta)\left(
 \|\Bu\|_{L^{2\beta}(S^{n-1})}^2
 +\|\Bf\|_{L^\beta(S^{n-1})}\right),\\
 \|\nabla p\|_{L^\beta(S^{n-1})}
 &\leq C(n,\beta)\left(
 \|\Bu\cdot\nabla\Bu\|_{L^\beta(S^{n-1})}
 +\|\Bf\|_{L^\beta(S^{n-1})}\right).
\end{align}
	\end{lemma}

 Throughout the rest of the paper, we fix this $p$ to be  a $(-2)$-homogeneous function.  
We use operators $\nabla^{S^{n-1}}$ and $\Delta_{S^{n-1}}$ to denote the spherical gradient and Laplace-Beltrami operator on the sphere. Norms of homogeneous functions or vector fields on the sphere are always referred to the corresponding norms of their restrictions on the unit sphere. To express the equations on the sphere, set $r=|x|$, $\Be_r=\frac{x}{r}$, and write
\begin{equation*}
 \Bu=\frac{\Bu^t+u^r\Be_r}{r},\qquad
 \Bf=\frac{\Bf^t+f^r\Be_r}{r^3},\qquad
 p=\frac{\bar p}{r^2}.
\end{equation*}
Here $\Bu^t$ and $u^r\Be_r$ are the tangential and radial components of $r\Bu$, while $\Bf^t$ and $f^r\Be_r$ are those of $r^3\Bf$. The functions $\Bu^t,u^r,\Bf^t,f^r$, and $\bar p$ are all $0$-homogeneous.
We use the following quantity throughout the paper,
\begin{equation}\label{eq:theta}
 \theta:=\frac{(n-2)(n-1)}{2(n-3)}
 =\frac n2+\frac{1}{n-3}>\frac n2.
\end{equation}
Let $\omega_n$ denote the volume of the unit ball in $\mathbb R^n$. 
In this section, $C$ denotes a positive constant depending only on $n$, unless otherwise specified.

The next two spherical estimates were also established in \cite{Hao25}, where the structural relation
	 between the radial velocity and the total head pressure on the sphere (see \cite[Equation (5)]{Hao25}) plays a crucial role.

\begin{lemma}[{\cite[Lemma 2.2]{Hao25}}]\label{halfenergyestimate}
Assume $n\geq6$. Every self-similar solution $\Bu\in C^2(\mathbb{R}^n\setminus\{0\})$ of \eqref{NS} satisfies
 \begin{equation}\label{eq:keydec-2}
     \begin{aligned}
       &\int_{S^{n-1}} |\nabla \Bu|^2 +(n-4)|\Bu|^2 d\sigma
    \leq 	\int_{S^{n-1}} \Bf\cdot \Bu + (n-4)H_{+} u^r_{-}\,  d\sigma + \frac{n-4}{2} \int_{S^{n-1}} f^r u^r_{+}\, d\sigma.
     \end{aligned}
 \end{equation}
\end{lemma}

The estimate of radial velocity  below is \cite[Equation (25)]{Hao25}, obtained there in the proof of \cite[Lemma 2.4]{Hao25}.
\begin{lemma}\label{lem:urestim}
Assume $n\geq6$. Every self-similar solution $\Bu\in C^2(\mathbb R^n\setminus\{0\})$ of \eqref{NS} satisfies
\begin{equation}\label{eq:L2estur}
 (n-2)\|u^r\|_{L^2(S^{n-1})}^2+2\|H_-\|_{L^1(S^{n-1})}
 \leq 2\|H_+\|_{L^1(S^{n-1})}+\|f^r\|_{L^1(S^{n-1})},
\end{equation}
where $H_-=-\min\{0,H\}$.
\end{lemma}

The total head pressure satisfies, in distributions,
\begin{align} \label{eq:headpre2}
	-\Delta \, H
	+ \Bu \cdot \nabla H
	= - \frac{1}{2} |\partial_i u_j - \partial_j u_i |^2
	+\Bf\cdot \Bu - \div \Bf \quad \text{in }\mathbb{R}^n \setminus \{0\}.
\end{align}
To remove the dimension restriction in \cite{Hao25},
the first key step is to rewrite the last two terms above into a divergence form  $\div 	\BmcF$ as mentioned in Section \ref{sec:mainidea}. We have
\begin{lemma}[Estimates for the total head pressure]\label{lem:fluxBernoulli}
	Assume $n\geq6$. Let $H\in C^1(\mathbb R^n\setminus\{0\})$ be $(-2)$-homogeneous, let $\Bu\in C^1(\mathbb R^n\setminus\{0\})$ be $(-1)$-homogeneous and divergence free, and let $\BmcF\in L^n_{\mathrm{loc}}(\mathbb R^n\setminus\{0\})$ be $(-3)$-homogeneous. Suppose that
	\begin{equation}\label{eq:BernoulliFluxIneq}
		-\Delta H+\Bu\cdot\nabla H\leq-\div\BmcF
		\quad\text{in }\mathbb R^n\setminus\{0\}
	\end{equation}
	in the sense of distributions. Then
	\begin{equation}\label{eq:Hbyflux}
	\mathcal{H} = \|H_+\|_{L^\theta(S^{n-1})}
		\leq C(n)\|\BmcF\|_{L^n(S^{n-1})}
	\end{equation}
	where $\theta$ is defined by \eqref{eq:theta}.
\end{lemma}

\begin{proof}
	Let
	\(
	\alpha=\frac{n-4}{2}\geq1.
	\)
	Choose a nonnegative radial function $\xi\in C_c^\infty(0,\infty)$ which is not identically zero. Since $\alpha\geq1$, the function $H_+^\alpha\xi$ is  Lipschitz and has compact support away from the origin.   On a fixed annulus, $\nabla H$ and $\Bu\cdot\nabla H$ are bounded and $\BmcF\in L^n\subset L^2$. Hence, $H_+^\alpha\xi$ is an admissible test function in \eqref{eq:BernoulliFluxIneq}. Testing \eqref{eq:BernoulliFluxIneq} by $H_+^\alpha\xi$,  using $\div\Bu=0$ and integration by parts, we have 
	\begin{equation}\label{eq:estiforH}
		\begin{aligned}
		\alpha\int_{\mathbb{R}^n \setminus\{0\}} H_+^{\alpha-1}|\nabla H_+|^2\xi\,dx
		\leq{}&\alpha\int_{\mathbb{R}^n \setminus\{0\}} \BmcF\cdot\nabla H_+\,H_+^{\alpha-1}\xi\,dx\\
		&+\int_{\mathbb{R}^n \setminus\{0\}} \BmcF\cdot\nabla\xi\,H_+^\alpha dx
		-\int_{\mathbb{R}^n \setminus\{0\}} H_+^\alpha\nabla H_+\cdot\nabla\xi\,dx\\
		&+\frac{1}{\alpha+1}\int_{\mathbb{R}^n \setminus\{0\}} H_+^{\alpha+1}\Bu\cdot\nabla\xi\,dx.
			\end{aligned}
	\end{equation}
	Now one observes that  the last three terms on the right-hand side above containing $\nabla\xi$ vanish. Indeed, $\BmcF H_+^\alpha$, $H_+^\alpha\nabla H_+$ and $H_+^{\alpha+1}\Bu$
are homogeneous of degree $-n+1$,
	using the spherical coordinates, each term is  a constant multiple of $\int_0^\infty\xi'(r)dr=0$. 
 Taking $ \BmcF\cdot\nabla\xi\,H_+^\alpha $ as an example, one has
	\begin{equation*}
		\begin{aligned}
	\int_{\mathbb{R}^n \setminus\{0\}} \BmcF\cdot\nabla\xi\,H_+^\alpha dx &= 	\int_{\mathbb{R}^n \setminus\{0\}} b\cdot\xi'(r)\,H_+^\alpha dx\\
	&=\int_{S^{n-1}} \int_{0}^{\infty}(b H_+^\alpha) (1, \sigma)\cdot\xi'(r)\,dr \,d\sigma =0.
		\end{aligned}
	\end{equation*}
     Here, we use $b$ to denote the radial component of $\BmcF$, and  we write $\BmcF$  as
	\begin{equation*}
		\BmcF=r^{-3}(\BG+b\Be_r),
	\end{equation*}
	where $\BG$ denotes tangential part of $\BmcF$.
	Consequently, \eqref{eq:estiforH} becomes
	\begin{equation}\label{eq:fullspacefluxenergy}
		\alpha\int_{\mathbb{R}^n \setminus\{0\}} H_+^{\alpha-1}|\nabla H_+|^2\xi\,dx
		\leq
		\alpha\int_{\mathbb{R}^n \setminus\{0\}} \BmcF\cdot\nabla H_+\,H_+^{\alpha-1}\xi\,dx.
	\end{equation}
	  After canceling the   factor $\int_0^\infty r^{-1}\xi(r)dr$ on both sides of \eqref{eq:fullspacefluxenergy} by using the homogeneities of $H$ and $\BmcF$, we can obtain
	\begin{align*}
		\alpha\int_{S^{n-1}}H_+^{\alpha-1}
		\left(|\nabla^{S^{n-1}}H_+|^2+4H_+^2\right)d\sigma
		\leq{}&\alpha\int_{S^{n-1}}
		\BG\cdot\nabla^{S^{n-1}}H_+\,H_+^{\alpha-1}d\sigma\nonumber\\
		&-2\alpha\int_{S^{n-1}}bH_+^\alpha d\sigma.
	\end{align*}
 A direct calculation using Young's inequality implies that
	\begin{equation*}
		\left\|H_+^{\frac{\alpha+1}{2}}\right\|_{H^1(S^{n-1})}^2
		\leq C(n)\int_{S^{n-1}}|\BmcF|^2H_+^{\alpha-1}d\sigma.
	\end{equation*}
 The Sobolev inequality on $S^{n-1}$ \cite{Aubin76}, together with
\(
		(\alpha+1)\frac{n-1}{n-3}=\theta = \frac{(n-2)(n-1)}{2(n-3)}\)
 yields
	\begin{equation}\label{eq:fluxSobolevH}
		\|H_+\|_{L^\theta(S^{n-1})}^{\alpha+1} = \mathcal H^{\alpha+1}
		\leq C(n)\int_{S^{n-1}}|\BmcF|^2H_+^{\alpha-1}d\sigma.
	\end{equation}
	Since
	\begin{equation*}
		\frac2n+\frac{\alpha-1}{\theta}<1,
	\end{equation*}
	an application of H\"older's inequality  gives
	\begin{equation*}
		\int_{S^{n-1}}|\BmcF|^2H_+^{\alpha-1}\,d\sigma
		\leq C(n)\|\BmcF\|_{L^n(S^{n-1})}^2\mathcal H^{\alpha-1}.
	\end{equation*}
This estimate holds for every $n\geq6$.
	Combining this with \eqref{eq:fluxSobolevH}, and dividing by $\mathcal H^{\alpha-1}$ when $\mathcal H>0$, proves \eqref{eq:Hbyflux}. The case $\mathcal H=0$ is obvious.
\end{proof}

\section{A priori estimates for the weighted energy and the Bernoulli source flux}\label{subsec:morrey}
In this section, we give two key a priori estimates. One is the Morrey-type control
of the velocity. Another is the estimate for the Bernoulli source flux.
We first adapt the weighted identity for the pressure and center-averaging argument from \cite{WangYang26} to control the energy on balls centered in a fixed annulus. 
We  mention that the  weighted identity of the pressure using the weight $|x-y|^{s+2}$
for $0<s<n-4$
already appeared in \cite{Frehse94, FrehseRuzicka94}. The key difference here is to choose the endpoint weight $|x-y|^{n-2}$ inspired by \cite{WangYang26}.

\begin{lemma}[Endpoint weighted estimate and Morrey type control of the velocity]\label{lem:endpointMorrey}
Assume $n\geq 6$. Let $\Bu$ be a self-similar solution to \eqref{NS} such that $\Bu\in C^2(\mathbb R^n\setminus\{0\})$. There exist constants $\frac{1}{12}>\rho_0>0$ and $C>0$, depending only on $n$, such that
\begin{equation}\label{eq:Morreyall}
    \sup_{\substack{\frac12<|z|<\frac32\\0<\rho\leq\rho_0}}
    \rho^{2-n}\int_{B_\rho(z)}|\Bu|^2dx
    \leq C\left(\mathcal E+\mathcal H+N_f\right).
\end{equation}
\end{lemma}

\begin{proof}
By Lemma \ref{Lemmapressure} and the Sobolev inequality on $S^{n-1}$,
\begin{equation}\label{eq:pauxendpoint}
    \|p\|_{L^{\frac{n-1}{n-3}}(S^{n-1})}
    \leq C\|\Bu\|_{L^{\frac{2(n-1)}{n-3}}(S^{n-1})}^2+CN_f
    \leq C(\mathcal E+N_f).
\end{equation}
Fix $0<R_0<\frac{1}{24}$ and a function $\zeta\in C_c^\infty(B_{2R_0}(0))$ such that $0\leq\zeta\leq1$ and $\zeta=1$ on $B_{R_0}(0)$. For $\frac13<|y|<\frac53$, set $\zeta^{(y)}(x)=\zeta(x-y)$. Thus
\begin{equation*}
    0\leq\zeta^{(y)}(x)\leq1,\qquad
    \zeta^{(y)}(x)=1\ \text{in }B_{R_0}(y),\quad
    |\nabla\zeta^{(y)}(x)|+|\nabla^2\zeta^{(y)}(x)|\leq C.
\end{equation*}
Taking divergence of \eqref{NS} yields
\begin{equation}\label{eq:pressureendpoint}
    -\Delta p=\partial_i\partial_j(u_i u_j)-\div\Bf.
\end{equation}
We test \eqref{eq:pressureendpoint} by the  kernel $\zeta^{(y)}(x)|x-y|^{4-n}$. To justify this rigorously, for $\varepsilon>0$ let
\begin{equation*}
    K_\varepsilon^{(y)}(x):=(|x-y|^2+\varepsilon^2)^{\frac{4-n}{2}}
\end{equation*}
and use $\zeta^{(y)}K_\varepsilon^{(y)}$ as a test function.  We have, uniformly in $\varepsilon$, that
\begin{equation*}
    |K_\varepsilon^{(y)}|\leq C|x-y|^{4-n},\qquad
    |\nabla K_\varepsilon^{(y)}|\leq C|x-y|^{3-n},\qquad
    |\nabla^2K_\varepsilon^{(y)}|\leq C|x-y|^{2-n}.
\end{equation*}
 After letting $\varepsilon \to 0^+$ and using
\begin{equation*}
    \partial_i\partial_j|x-y|^{4-n}
    =(4-n)\delta_{ij}|x-y|^{2-n}
    +(4-n)(2-n)\frac{(x_i-y_i)(x_j-y_j)}{|x-y|^n},
\end{equation*}
one obtains
\begin{align}\label{eq:endpointweightedidentity}
    &(n-4)(n-2)\int_{\mathbb R^n}
    \frac{|\Bu\cdot(x-y)|^2\zeta^{(y)}}{|x-y|^n}\,dx
    -2(n-4)\int_{\mathbb R^n}
    \frac{H\zeta^{(y)}}{|x-y|^{n-2}}\,dx
    =\mathcal R(y),
\end{align}
where $K^{(y)}(x)=|x-y|^{4-n}$, and
\begin{align*}
 \mathcal R(y)={}&-\int_{\mathbb R^n}
 \left(2u_i u_j\partial_i\zeta^{(y)}\partial_jK^{(y)}
       +u_i u_jK^{(y)}\partial_i\partial_j\zeta^{(y)}\right)dx\\
 &-\int_{\mathbb R^n}p
 \left(2\nabla\zeta^{(y)}\cdot\nabla K^{(y)}+K^{(y)}\Delta\zeta^{(y)}\right)dx
 -\int_{\mathbb R^n}\Bf\cdot\nabla(\zeta^{(y)}K^{(y)})\,dx.
\end{align*}
The derivatives of $\zeta^{(y)}$ are supported in a domain with a fixed positive distance from $y$, and all supports lie in $\{\frac{1}{4}<|x|<\frac{7}{4}\}$. The force term is bounded by $CN_f$, since $|\nabla K^{(y)}|\leq C|x-y|^{3-n}$ is locally integrable. Using \eqref{eq:pauxendpoint} and homogeneity, we obtain
\begin{equation}\label{eq:Rendpoint}
    |\mathcal R(y)|\leq C(\mathcal E+N_f).
\end{equation}

Since $H\leq H_+$, \eqref{eq:endpointweightedidentity} and \eqref{eq:Rendpoint} imply
\begin{equation}\label{eq:directionalall}
    \int_{\mathbb R^n}
    \frac{|\Bu\cdot(x-y)|^2\zeta^{(y)}}{|x-y|^n}\,dx
    \leq C\int_{\mathbb R^n}\frac{H_+\zeta^{(y)}}{|x-y|^{n-2}}\,dx
    +C(\mathcal E+N_f).
\end{equation}
Recall \eqref{eq:theta} that $\theta>\frac{n}{2}$,  therefore $(n-2)\theta'<n$, where $\theta'$ is H\"older conjugate of $\theta$ such that $\frac{1}{\theta}+\frac{1}{\theta'}=1$. H\"older's inequality and the $(-2)$-homogeneity of $H$ give
\begin{equation}\label{eq:Hpotentialall}
    \sup_{\frac13<|y|<\frac53}
    \int_{\mathbb R^n}\frac{H_+\zeta^{(y)}}{|x-y|^{n-2}}\,dx
    \leq C\mathcal H.
\end{equation}
Consequently,
\begin{equation}\label{eq:directionalall2}
    \int_{\mathbb R^n}
    \frac{|\Bu\cdot(x-y)|^2\zeta^{(y)}}{|x-y|^n}\,dx
    \leq C(\mathcal E+\mathcal H+N_f)
\end{equation}
uniformly for $\frac13<|y|<\frac53$.

We now average the above with respect to $y$. Choose $0<\rho_0<\min\{R_0,\frac{1}{12}\}$. For $\frac12<|z|<\frac32$ and $0<\rho\leq\rho_0$, integrate \eqref{eq:directionalall2} over $y\in B_{2\rho}(z)$.  If $x\in B_\rho(z)$ and $y\in B_\rho(x)$, then $y\in B_{2\rho}(z)$ and $\zeta^{(y)}(x)=1$. The integrand is measurable in $(x,y)$, so Tonelli's theorem gives
\begin{align*}
    C(\mathcal E+\mathcal H+N_f)\rho^n
    &\geq \int_{B_{2\rho}(z)}\int_{\mathbb R^n}
    \frac{|\Bu(x)\cdot(x-y)|^2\zeta^{(y)}(x)}{|x-y|^n}\,dx\,dy\\
    &\geq \int_{B_\rho(z)}\int_{B_\rho(x)}
    \frac{|\Bu(x)\cdot(x-y)|^2}{|x-y|^n}\,dy\,dx.
\end{align*}
For every $a\in\mathbb R^n$, it holds that
\begin{equation*}
    \int_{B_\rho(0)}\frac{|a\cdot y|^2}{|y|^n}\,dy=c_n\rho^2|a|^2
\end{equation*}
for some $c_n>0$. Therefore, we have
\begin{equation*}
    \rho^{2-n}\int_{B_\rho(z)}|\Bu|^2dx
    \leq C(\mathcal E+\mathcal H+N_f).
\end{equation*}
This proves \eqref{eq:Morreyall} and thus the lemma.
\end{proof}

We next estimate the divergence-form source $\BmcF$ defined in \eqref{eq:defF}. Before that,  we recall the following potential estimate needed in the proof.
\begin{lemma}[{\cite[Lemma 2.4]{WangYang26}}]\label{lem:morreyBMO}
Let $n\geq3$, $n_*:=\frac{2n}{n+2}$, and let $g\in L^{n_*}(\mathbb R^n)$ satisfy
\begin{equation}\label{definitionm_1}
    m_1(g):=\sup_{z\in\mathbb R^n,\rho>0}
    \rho^{1-n}\int_{B_\rho(z)}|g(x)|\,dx<\infty.
\end{equation}
Then one has
\begin{equation}\label{eq:morreyBMOgeneral}
    \|\nabla(-\Delta)^{-1}g\|_{L^2(\mathbb R^n)}\leq C(n)\|g\|_{L^{n_*}(\mathbb R^n)}
    \quad \text{and}\quad
    \|\nabla(-\Delta)^{-1}g\|_{\operatorname{BMO}(\mathbb R^n)}\leq C(n)m_1(g).
\end{equation}
Here, for a locally integrable function \(f \in L^1_{\text{loc}}(\mathbb{R}^n)\), the BMO seminorm is defined as
\[\|f\|_{\operatorname{BMO}(\mathbb R^n)}=\sup _{B}\frac{1}{|B|}\int _{B}|{}f(x)-f_{B}|{}\,dx\]
where the supremum is taken over all balls \(B \subset \mathbb{R}^n\), and \(f_B = \frac{1}{|B|} \int_B f(y) \, dy\).
\end{lemma}

\begin{lemma}[Estimate for the Bernoulli source flux]\label{lem:sourcefluxSS}
Assume $n\geq6$ and let $\Bu$ be as in Lemma \ref{lem:endpointMorrey}. Define
\begin{equation}\label{eq:defgPsiF}
    g:=|\Bf||\Bu|,\quad
    \Psi:=(-\Delta)^{-1}g,\quad \text{and}\quad
    \BmcF:=\nabla\Psi+\Bf.
\end{equation}
Then $\Psi$ is $(-2)$-homogeneous, $\BmcF$ is $(-3)$-homogeneous, and that
\begin{equation}\label{eq:fluxdivergence}
    -\div\BmcF=g-\div\Bf.
\end{equation}
Moreover, for every fixed $2<q<\infty$, it holds that
\begin{equation}\label{eq:fluxLq}
    \|\BmcF\|_{L^q(S^{n-1})}
    \leq C(n,q)N_f\left(1+N_f^{\frac{1}{2}}+\mathcal E^{\frac{1}{2}}+\mathcal H^{\frac{1}{2}}\right).
\end{equation}
\end{lemma}

\begin{proof} {\it Step 1.}
Since $g$ is $(-4)$-homogeneous, the Newtonian potential
\begin{equation}\label{eq:Psidefinition}
    \Psi(x)=\frac{1}{n(n-2)\omega_n}\int_{\mathbb R^n}\frac{g(y)}{|x-y|^{n-2}}dy
\end{equation}
is absolutely convergent for $x\neq0$. Indeed, near the origin the  integral is bounded by $C\int_0^1r^{n-5}dr$, while at infinity it is bounded by $C\int_1^\infty r^{-3}dr$. The same argument applies to $\nabla\Psi$, whose kernel is also locally integrable near $x$. 

A change of variables in \eqref{eq:Psidefinition} gives
\begin{equation*}
    \Psi(\lambda x)=\lambda^{-2}\Psi(x)\quad
  \text{and}\quad  \nabla\Psi(\lambda x)=\lambda^{-3}\nabla\Psi(x).
\end{equation*}
Since $g$ is locally Lipschitz away from the origin, local  estimates for Poisson equation give $\Psi\in C^{2,\delta}_{\mathrm{loc}}(\mathbb R^n\setminus\{0\})$ for every $0<\delta<1$. Moreover, $-\Delta\Psi=g$, and hence \eqref{eq:fluxdivergence} follows.

{\it Step 2.} Let $\psi=\Psi|_{S^{n-1}}$. Then using the Laplace–Beltrami operator on $S^{n-1}$  and that $\Psi$ is $(-2)$-homogeneous, we have
\begin{equation}\label{eq:spherePsi}
    \left(-\Delta_{S^{n-1}}+2(n-4)\right)\psi=g|_{S^{n-1}}.
\end{equation}
Testing \eqref{eq:spherePsi} by $\psi$ and using
\begin{equation*}
    \|g\|_{L^2(S^{n-1})}\leq N_f\|\Bu\|_{L^2(S^{n-1})}\leq CN_f\mathcal E^{\frac{1}{2}}
\end{equation*}
yield
\begin{equation}\label{eq:PsiH1}
    \|\psi\|_{H^1(S^{n-1})}\leq CN_f\mathcal E^{\frac{1}{2}}.
\end{equation}
Consequently, on every fixed annulus $A\Subset\mathbb R^n\setminus\{0\}$,
\begin{equation}\label{eq:PsiL2annulus}
    \|\nabla\Psi\|_{L^2(A)}\leq C(n,A)N_f\mathcal E^{\frac{1}{2}}.
\end{equation}

Choose fixed annuli
\begin{equation*}
    A_0:=\left\{\frac78<|x|<\frac98\right\}
    \Subset
    A_1:=\left\{\frac34<|x|<\frac54\right\}
    \Subset
    A_2:=\left\{\frac23<|x|<\frac43\right\},
\end{equation*}
and let $\chi\in C_c^\infty(A_2)$ satisfy $0\leq\chi\leq1$ and $\chi=1$ on $A_1$. Set
\begin{equation*}
    \Phi:=(-\Delta)^{-1}(\chi g)\quad\text{and}\quad  \Bw:=\nabla\Phi.
\end{equation*}
For $\rho_0$ appeared in Lemma \ref{lem:endpointMorrey}, if $0<\rho\leq\rho_0$ and $B_\rho(z)$ intersects $\supp\chi$, then $\frac{1}{2}<|z|<\frac{3}{2}$ due to $\rho_0<\frac{1}{12}$.
 It follows from Lemma \ref{lem:endpointMorrey} and H\"older's inequality that one has
\begin{align*}
    \int_{B_\rho(z)}\chi g\,dx
    &\leq CN_f|B_\rho(z)|^{\frac{1}{2}}
\left(\int_{B_\rho(z)}|\Bu|^2dx\right)^{\frac{1}{2}}\leq CN_f(\mathcal E+\mathcal H+N_f)^{\frac{1}{2}}\rho^{n-1}.
\end{align*}
 If $\rho \geq \rho_0$, then applying  the H\"older's inequality gives 
\begin{align*}
	\int_{B_\rho(z)} \chi g \, dx \leq C N_f |A_2|^{\frac12} \left(\int_{A_2}|\Bu|^2dx\right)^{\frac{1}{2}} \leq C N_f \mathcal E^\frac12. 
\end{align*}
 Therefore, it holds that
\begin{equation}\label{eq:m1gSS}
    m_1(\chi g)\leq CN_f(\mathcal E+\mathcal H+N_f)^{\frac{1}{2}}.
\end{equation}
Here $m_1$ is defined by \eqref{definitionm_1}.
Furthermore, since $n_*=2n/(n+2)<2$ and $\chi g$ has support in a fixed annulus, one has
\begin{equation}\label{eq:gqstarSS}
    \|\chi g\|_{L^{n_*}(\mathbb R^n)}\leq \|\chi g\|_{L^{2}(\mathbb R^n)} \leq  CN_f\mathcal E^{\frac{1}{2}}.
\end{equation}

{\it Step 3.}
Applying Lemma \ref{lem:morreyBMO} yields
\begin{equation}\label{eq:wBMO}
    \|\Bw\|_{L^2(\mathbb R^n)}\leq CN_f\mathcal E^{\frac{1}{2}}
    \quad
    \textrm{and}  \quad \|\Bw\|_{\operatorname{BMO}(\mathbb R^n)}
    \leq CN_f(\mathcal E+\mathcal H+N_f)^{\frac{1}{2}}.
\end{equation}
Since $\chi=1$ on $A_1$, the function $\eta:=\Psi-\Phi$ is harmonic in $A_1$. By \eqref{eq:PsiL2annulus}, \eqref{eq:wBMO}, and the interior estimates for harmonic functions, one has
\begin{equation}\label{eq:etaharmonic}
    \|\nabla\eta\|_{L^\infty(A_0)} \leq C \|\nabla \Psi -\nabla \Phi\|_{L^2(A_1)} \leq CN_f\mathcal E^{\frac{1}{2}}.
\end{equation}
Let $B$ be a fixed ball containing $A_0$. By the John-Nirenberg inequality \cite[Chapter IV]{Stein93} and the H\"older's inequality, we have
\begin{equation*}
    \|\Bw\|_{L^q(A_0)}
    \leq C(n,q)\left(\|\Bw\|_{\operatorname{BMO}(\mathbb R^n)}+\|\Bw\|_{L^2(\mathbb R^n)}\right) \leq C (n, q) N_f \left( \mathcal E^{\frac12} + \mathcal H^{\frac12} +  N_f^{\frac12}\right).
\end{equation*}
Indeed, this can be obtained by applying the John-Nirenberg inequality  to $\Bw-\Bw_B$ and using $|\Bw_B|\leq C\|\Bw\|_{L^2(A_0)}$, where $\Bw_B$ is the average of $\Bw$ on $B$. Combining this estimate with \eqref{eq:wBMO}, \eqref{eq:etaharmonic}, and $\BmcF=\Bw+\nabla\eta+\Bf$ on $A_0$, we obtain
\begin{equation}\label{eq:fluxannulus}
    \|\BmcF\|_{L^q(A_0)}
    \leq C(n,q)N_f\left(1+N_f^{\frac{1}{2}}+\mathcal E^{\frac{1}{2}}+\mathcal H^{\frac{1}{2}}\right).
\end{equation}
Finally, since $\BmcF$ is $(-3)$-homogeneous,
\begin{equation*}
    \|\BmcF\|_{L^q(A_0)}^q
    =\left(\int_{\frac{7}{8}}^{\frac{9}{8}}r^{n-1-3q}dr\right)
    \|\BmcF\|_{L^q(S^{n-1})}^q.
\end{equation*}
One then obtains \eqref{eq:fluxLq} from \eqref{eq:fluxannulus}.
\end{proof}

\section{Proof of Theorem \ref{thm:main}}\label{sec:mainresult}

The goal of this section is to prove the following proposition and  the main result, Theorem \ref{thm:main}.
\begin{pro}[Energy estimates]\label{Energy estimates2}
Assume $n\geq6$. Let $\Bu$ be a self-similar solution to \eqref{NS} such that $\Bu\in C^2(\mathbb R^n\setminus\{0\})$. Then we have
\begin{equation}\label{eq:energyes3}
       \mathcal E =  \int_{S^{n-1}}\left(|\nabla\Bu|^2+(n-4)|\Bu|^2\right)d\sigma
    \leq C(n)\left(N_f^{\frac32}+N_f^6\right),
\end{equation}
\begin{equation}\label{integralofH}
   \mathcal{H} =  \|H_+\|_{L^\theta(S^{n-1})}
    \leq C(n)\left(N_f+N_f^4\right),
\end{equation}
and
\begin{equation}\label{eq:energyes4}
    \|p\|_{L^{\frac{n-1}{n-3}}(S^{n-1})}
    +\|\nabla p\|_{L^{\frac{n-1}{n-2}}(S^{n-1})}
    \leq C(n)\left(N_f+N_f^6\right),
\end{equation}
where $N_f$ is defined in \eqref{defunifiedquantities}.
In particular, all the quantities on the left-hand sides above tend to zero as $N_f\to0$.
\end{pro}
\begin{proof}
We first prove the following estimate, 
\begin{equation}\label{eq:Hclosureunified}
    \mathcal H\leq C\left(N_f+N_f^2+N_f\mathcal E^{\frac{1}{2}}\right),
\end{equation}
which is valid for every $n\geq6$.
It follows from \eqref{eq:headpre2} and \eqref{eq:defgPsiF} that
\begin{equation*}
    -\Delta H+\Bu\cdot\nabla H
    \leq |\Bf||\Bu|-\div\Bf
    =-\div\BmcF.
\end{equation*}
Applying Lemmas \ref{lem:fluxBernoulli} and \ref{lem:sourcefluxSS} with $q=n$ gives
\begin{equation*}
    \mathcal H
    \leq CN_f\left(1+N_f^{\frac{1}{2}}+\mathcal E^{\frac{1}{2}}+\mathcal H^{\frac{1}{2}}\right).
\end{equation*}
By Young's inequality, we obtain \eqref{eq:Hclosureunified}.

We now close the energy estimates as follows. Lemma \ref{lem:urestim} and H\"older's inequality give
\begin{equation}\label{eq:urunified}
    \|u^r\|_{L^2(S^{n-1})}^2\leq C(\mathcal H+N_f).
\end{equation}
By Lemma \ref{halfenergyestimate}, H\"older's inequality, $\theta'<2$ (recall $\theta'$ is the H\"older conjugate of $\theta$), and \eqref{eq:urunified}, one has
\begin{align}\label{eq:Eclosureunified}
    \mathcal E
    &\leq CN_f\mathcal E^{\frac{1}{2}}
    +C\mathcal H\|u^r\|_{L^2(S^{n-1})}
    +CN_f\|u^r\|_{L^2(S^{n-1})}
    \leq CN_f\mathcal E^{\frac{1}{2}}+C(\mathcal H+N_f)^{\frac{3}{2}}.
\end{align}
Substituting \eqref{eq:Hclosureunified} into \eqref{eq:Eclosureunified} yields
\begin{equation*}
    \mathcal E
    \leq CN_f\mathcal E^{\frac{1}{2}}
    +C\left(N_f+N_f^2+N_f\mathcal E^{\frac{1}{2}}\right)^{\frac{3}{2}}.
\end{equation*}
Using the inequality $(a+b+c)^{\frac{3}{2}}\leq C(a^{\frac{3}{2}}+b^{\frac{3}{2}}+c^{\frac{3}{2}})$ and 
\begin{equation*}
    CN_f\mathcal E^{\frac{1}{2}}\leq\frac14\mathcal E+CN_f^2,
    \quad
    CN_f^{\frac{3}{2}}\mathcal E^{\frac{3}{4}}\leq\frac14\mathcal E+CN_f^6,
\end{equation*}
it follows that
\begin{equation*}
    \mathcal E\leq C\left(N_f^{\frac{3}{2}}+N_f^2+N_f^3+N_f^6\right)
    \leq C\left(N_f^{\frac{3}{2}}+N_f^6\right).
\end{equation*}
This proves \eqref{eq:energyes3}. Substituting \eqref{eq:energyes3} into \eqref{eq:Hclosureunified}, 
we obtain
\begin{equation*}
    \mathcal H\leq C(N_f+N_f^4),
\end{equation*}
which is \eqref{integralofH}.

For the pressure, \eqref{eq:pauxendpoint} already gives $\|p\|_{L^{\frac{n-1}{n-3}}(S^{n-1})}\leq C(\mathcal E+N_f)$.
Applying \eqref{eq:pressure-potential-estimates} and the H\"older's and Sobolev inequalities, we get
\begin{align*}
    \|\nabla p\|_{L^{\frac{n-1}{n-2}}(S^{n-1})}
    &\leq C\|\Bu\cdot\nabla\Bu\|_{L^{\frac{n-1}{n-2}}(S^{n-1})}+CN_f\\
    &\leq C\|\Bu\|_{L^{\frac{2(n-1)}{n-3}}(S^{n-1})}
    \|\nabla\Bu\|_{L^2(S^{n-1})}+CN_f\\
    &\leq C\mathcal E+CN_f.
\end{align*}
This, together with \eqref{eq:energyes3}, gives
\eqref{eq:energyes4}. Hence the proof of Proposition \ref{Energy estimates2} is completed.
\end{proof}

\begin{proof}[Proof of Theorem \ref{thm:main}]
	With
	a priori estimate established in Proposition~\ref{Energy estimates2}, the rest of the proof is
	the same as \cite[Section~3]{Hao25}.
	In particular, the construction and compactness of the Stokes map,
	the Leray-Schauder  argument, the regularity,
	and the small-data uniqueness argument can be carried out without any change.
	We therefore omit these  details and refer the reader to
	\cite[Section~3]{Hao25} for the proof of Theorem \ref{thm:main}.
\end{proof}

\medskip

{\bf Acknowledgements.} 
The research of Gui is supported by  NSFC Key Program (Grant No. 12531010),  University of Macau research grants CPG2024-00016-FST, CPG2025-00032-FST, CPG2026-00027-FST, SRG2023-00011-FST, MYRG-GRG2023-00139-FST-UMDF, UMDF Professorial Fellowship of Mathematics, Macao SAR FDCT 0003/2023/RIA1 and  Macao SAR FDCT 0024/2023/RIB1.
The research of Wang is partially supported by NSFC Grant 12671279, and the Natural Science Foundation of Jiangsu Province (Grant~BK20240147).
The research of Xie is partially supported by NSFC Grants 12571238 and 12426203.

	\bibliographystyle{abbrv}
	\bibliography{steadyNS}
\end{document}